\documentclass[11pt]{article}

\usepackage{amsmath,amssymb,graphicx,amsthm,hyperref}

\usepackage[margin=1.3in]{geometry}

\hypersetup{
pdfauthor = {Pavle Blagojevic, Boris Bukh, Roman Karasev},
pdftitle = {Turan numbers for K\_s,t-free graphs: topological obstructions and algebraic constructions},
pdfsubject = {Mathematics},
pdfkeywords = {Turan numbers, Zarankiewicz problem, grids, hypersurfaces},
pdfstartview = {FitH},
pdfpagemode = {None}}

\author{Pavle V. M. Blagojevi\'{c}\thanks{%
The research leading to these results has received funding from the European Research
Council under the European Union's Seventh Framework Programme (FP7/2007-2013) /
ERC Grant agreement no.~247029-SDModels. Also supported by the grant ON 174008 of the Serbian
Ministry of Education and Science.} \\
\small{\url{pavleb@mi.sanu.ac.rs}}\footnote{Mathemati\v cki Institut SANU, Knez Michailova 36, 11001 Beograd, Serbia}
\and Boris Bukh \\
\small{\url{B.Bukh@dpmms.cam.ac.uk}}\footnote{Centre for Mathematical Sciences, Cambridge CB3 0WB, England, and
Churchill College, Cambridge CB3 0DS, England} \and Roman~Karasev \thanks{Supported by the Dynasty Foundation, the          
President's of Russian Federation grant MD-352.2012.1, the Russian Foundation for         
Basic Research grants 10-01-00096 and 10-01-00139, the Federal Program ``Scientific       
and scientific-pedagogical staff of innovative Russia'' 2009--2013, and the Russian       
government project 11.G34.31.0053.}\\
\small{\url{r_n_karasev@mail.ru}
\footnote{Dept. of Mathematics, Moscow Institute of Physics and             
Technology, Institutskiy per. 9, Dolgoprudny, Russia 141700, and Laboratory of Discrete and Computational Geometry, Yaroslavl'
State University, Sovetskaya st. 14, Yaroslavl', Russia 150000}}}
\date{}

\theoremstyle{plain}
\newtheorem{theorem}{Theorem}
\newtheorem{lemma}[theorem]{Lemma}
\newtheorem{corollary}[theorem]{Corollary}

\newtheorem{proposition}[theorem]{Proposition}

\newtheorem*{problem}{Problem}

\DeclareMathOperator{\ex}{ex}
\DeclareMathOperator{\sh}{sh}                                  % shape
\DeclareMathOperator{\supp}{supp}                              % support
\DeclareMathOperator{\conv}{conv}                              % convex hull
\DeclareMathOperator{\vol}{vol}                                % volume
\DeclareMathOperator{\Prj}{Prj}                                % the space of all projections

\newcommand*{\eqdef}{\stackrel{\text{\tiny def}}{=}}           % definition by equality

\newcommand*{\A}{\mathbb{A}}                                   % affine space
\newcommand*{\R}{\mathbb{R}}                                   % real numbers
\newcommand*{\C}{\mathbb{C}}                                   % complex numbers
\newcommand*{\Z}{\mathbb{Z}}                                   % integers
\newcommand*{\Fp}{\mathbb{F}_p}                                % finite field with p elements
\newcommand*{\bx}{\mathbf{x}}                                  
\newcommand*{\Fpm}{\mathbb{F}_{p^m}}                           % finite field with p^m elements
\newcommand*{\abs}[1]{\lvert #1\rvert}                         % absolute value, cardinality
\title{Tur\'an numbers for $K_{s,t}$-free graphs: topological obstructions and algebraic constructions}

\begin{document}
\maketitle

\begin{abstract}
We show that every hypersurface in $\R^s\times \R^s$ contains a large grid, i.e., the set
of the form $S\times T$, with $S,T\subset \R^s$. We use this to deduce that
the known constructions of extremal $K_{2,2}$-free and $K_{3,3}$-free graphs
cannot be generalized to a similar construction of $K_{s,s}$-free graphs for any $s\geq 4$.
We also give new constructions of extremal $K_{s,t}$-free graphs for large $t$.   
\end{abstract}

\section{Introduction}

For a graph $H$ the Tur\'an number $\ex(H,n)$ is the maximum number
of edges that a graph on $n$ vertices can have without containing
a copy of $H$. Erd\H{o}s--Stone theorem \cite{erdos_stone} asserts that
$\ex(H,n)=\left(1-\frac{1}{\chi(H)-1}\right)\binom{n}{2}+o(n^2)$, where
$\chi(H)$ is the chromatic number of $H$. If $\chi(H)\geq 3$, this result
is an asymptotic formula for $\ex(H,n)$, but if $G$ is bipartite, the
order of magnitude of $\ex(H,n)$ is in general open.

Most of the research on $\ex(H,n)$ for bipartite $H$ has focused on
two classes of graphs: complete bipartite graphs $K_{s,t}$, and even
cycles $C_{2t}$. We shall briefly review both cases. Suppose $G=(V,E)$
is a $K_{s,t}$-free graph with $s\leq t$. The
inequality $\sum_{x\in V} \binom{d(x)}{s}\leq (t-1)\binom{n}{s}$ \cite{kovari_sos_turan}
implies the simplest upper bound of
$\ex(K_{s,t},n)\leq \tfrac{1}{2}\sqrt[s]{t-1} n^{2-1/s}+o(n^{2-1/s})$
due to K{\"o}vari--S{\'o}s--Tur{\'a}n.
The bound has been improved by F\"uredi \cite{furedi_zarankiewicz} to
\begin{equation*}
\ex(K_{s,t},n)\leq \tfrac{1}{2}\sqrt[s]{t-s+1}n^{2-1/s}+o(n^{2-1/s}).
\end{equation*}

The only cases where the upper bound has been matched by a construction
with $\Omega(n^{2-1/s})$ edges are
\begin{align*}
\ex(K_{2,2},n)&=\tfrac{1}{2}n^{3/2}+o(n^{3/2})\qquad&\text{Erd\H{o}s--R\'enyi--S\'os \cite{erdos_renyi_sos} and Brown \cite{brown_construction}},\\
\ex(K_{2,t},n)&=\tfrac{1}{2}\sqrt{t-1}\,n^{3/2}+o(n^{3/2})\qquad&\text{F\"uredi \cite{furedi_ktwot}},\\
\ex(K_{3,3},n)&=\tfrac{1}{2}n^{5/3}+o(n^{5/3}),&\text{Brown \cite{brown_construction}},\\
\ex(K_{s,t},n)&\geq c_{s,t}n^{2-1/s},&\text{if $t\geq s!+1$, Koll\'ar--R\'onyai--Szab\'o \cite{kollar_ronyai_szabo}},\\
\ex(K_{s,t},n)&\geq c_{s,t}n^{2-1/s},&\text{if $t\geq (s-1)!+1$, Alon--R\'onyai--Szab\'o \cite{alon_ronyai_szabo}}.
\end{align*}

The constructions cited above are similar to one another.
Each of them is based on an algebraic
hypersurface\footnote{Algebraic hypersurface is a variety of codimension $1$.}
in $\Fp^s\times \Fp^s$ of bounded degree, such that $V$
\emph{contains neither an $s$-by-$t$ grid nor $t$-by-$s$ grid}.
An $s$-by-$t$ grid in $V$ is a subset of the form
$S\times T$, where $S,T\subset \A^s$ and $\abs{S}=s$, $\abs{T}=t$.
A hypersurface without $s$-by-$t$ grid, and $t$-by-$s$ grid
gives rise to a $K_{s,t}$-free graph with $c_{s,t} n^{2-1/m}$.
Indeed, the bipartite graph with both edge classes being $\Fp^m$,
and edge set $E=\{ (x,y) \in V(\Fp^m) : x,y\in \Fp^m \}$ contains
no $K_{s,t}$. As $V$ is without loss irreducible, the estimate
$\abs{E}=p^{\dim V}(1+O(1/\sqrt{p}))$ holds by \cite{lang_weil}.

The defining equations of the hypersurfaces that have been employed to construct
extremal $K_{s,t}$-free graphs are
\begin{align}
x_1y_1+x_2y_2=1,\qquad&\text{for $K_{2,2}$},\label{constr_two}\\
(x_1-y_1)^2+(x_2-y_2)^2+(x_3-y_3)^2=1,\qquad&\text{for $K_{3,3}$},\label{constr_three}\\
N_s(x+y)=1,\qquad&\text{for $K_{s,t}$, $t\geq s!+1$},\label{constr_kollar}\\
N_{s-1}(x'+y')=x_1y_1,\qquad&\text{for $K_{s,t}$, $t\geq (s-1)!+1$},\label{constr_alon}
\end{align}
where $N_m(x)$ is the norm form of the field extension $\Fpm/\Fp$, and the notation
$x'$ stands for projection of the vector $x=(x_1,x_2,\dotsc,x_m)$ on the last $m-1$
components. F\"uredi's construction
for $K_{2,t}$-free graphs also relies on the fact that the surface $x_1y_1+x_2y_2=1$ contains
no $2$-by-$2$ grid. A recent result by Ball and Pepe \cite{ball_pepe} shows that
the construction \eqref{constr_alon} for $s=4$ also does not contain $K_{5,5}$,
giving the best known construction for $K_{5,5}$-free graphs.

If $V$ has integer coefficients, and $V$ contains no $s$-by-$t$ grid over $\C$,
then $V$ contains no $s$-by-$t$ grid over $\Fp$ for every sufficiently
large $p$ thanks to the following standard result:
\begin{proposition}
If $U$ and $W$ are affine varieties defined over $\Z$, and
$U\subseteq W$ over $\C$, then $U\subseteq W$ over $\Fp$
for every sufficiently large prime $p$.
\end{proposition}
\begin{proof}
Let $U$ be defined by equations $f_i(x)=0$ with integer coefficients,
and $W$ be defined by equations $g_j(x)=0$ with integer coefficients.
By Hilbert's Nullstellensatz for each $j$ we have an equality
\[
  \sum_i h_{i,j}(x) f_i(x) = g_j(x)^r
\]
for some polynomials $h_{i,j}$ and a positive integer $r$. As these 
are linear conditions on $h_{i,j}(x)$'s, they can be chosen with rational coefficients.
Multiplying by a common denominator $N$ we obtain equalities for polynomials with
integer coefficients:
\[
  \sum_i H_{i,j}(x) f_i(x) = N g_j(x)^r,
\]
which implies $U\subseteq W$ over fields of characteristic greater than $N$.
\end{proof}
For our application, we write $V=\{f=0\}$, let $U=\{(x,y)\in (\A^s)^s\times (\A^s)^t : f(x_i,y_j)=0\}$,
and $W=\bigcup_{i,i'} \bigl\{(x,y) : x_i=x_{i'}\bigr\}\cup \bigcup_{j,j'} \bigl\{(x,y) : y_j=y_{j'}\bigr\}$.
Then $U\subseteq W$ over field $K$ if and only if the hypersurface $V$ contains no
$s$-by-$t$ grid over $K$. So, for example the variety \eqref{constr_two} contains 
no $2$-by-$2$ grid even over $\C$, and hence over $\Fp$ for sufficiently large $p$.

The variety \eqref{constr_three} has a weaker property. Whereas
it does contain $3$-by-$3$ grid over $\C$, it does not contain
a $3$-by-$3$ grid over $\R$.

The analogous situation also occurs for $\ex(C_{2t},n)$.
The upper bound $\ex(C_{2t},n)\leq 100 t n^{1+1/t}$
by Bondy--Simonovits \cite{bondy_simonovits} has been matched only for
 $t=2,3,5$. The case $t=2$ was mentioned
above under the guise of $\ex(K_{2,2},n)$. The constructions in cases $t=3,5$ \cite{benson_constr,wenger_constr}
are also algebraic. Namely, there is a $(t+1)$-dimensional variety $V\subset \A^t\times \A^t$
such there do not exist sets $\{x_1,\dotsc,x_t\}$ and $\{y_1,\dotsc,y_t\}$ satisfying $(x_i,y_j)\in V$
whenever $i=j$ or $j=i+1 \pmod t$. For example, Wenger's construction \cite{wenger_constr} is based on
the variety given by the equations
\begin{equation}\label{constr_wenger}
y_j=x_j+x_{j+1}y_{t},\qquad\text{for }j=1,\dotsc,t-1.
\end{equation}

It is thus natural to ask whether there are similar algebraic constructions
of $G$-free graphs, for other
graphs $G$. Since constructions \eqref{constr_two}, \eqref{constr_three}, \eqref{constr_wenger} were
based on the constructions that worked not only over $\Fp$, but also over $\R$, it is natural
to look for their extensions with the same property. Our first result rules out such constructions
for $K_{s,s}$-free graphs.
\begin{theorem}\label{thm_maintwo}
Suppose $p$ is a prime, and $d=2\bigl\lceil\frac{p-1}{4}\bigr\rceil+2$. Then every continuous function
$f\colon \R^d\times \R^d\to \R$ is constant on some $p$-by-$p$ grid, i.e., there
exist sets $X,Y\subset\R^d$ of size $\abs{X}=\abs{Y}=p$ such that $f$ is constant on $X\times Y$.
\end{theorem}

\begin{corollary}
Let $p$ and $d$ as above. Then every smooth hypersurface $V\subset \R^d\times \R^d$ contains
a $p$-by-$p$ grid.
In particular, every hypersurface $V\subset \R^4\times \R^4$ contains a $5$-by-$5$ grid.
\end{corollary}

{\renewcommand{\proofname}{Proof of the corollary assuming the theorem.}
\begin{proof}
We may assume that $0\in V$ and that unit vector $u$ is
the normal to $V$ at $0$. Assume that in a neighborhood of $0$, the variety $V$
is given by the equation $f=0$. By the inverse function theorem there exists a smooth function
$g$ on the neighborhood of $0$ such that $f(x-g(x)v)=0$. Since a neighborhood of $0$ is
homeomorphic to $\R^d\times\R^d$, by preceding theorem there are sets $X,Y\subset \R^d$ of size $p$
such that $g(x,y)=c$ for every $p\in X_1\times X_2$. Thus $f$ vanishes on $X_1\times X_2+cv$.
Since the latter set is a $p$-by-$p$ grid, the lemma follows.
\end{proof}}
Both the theorem and the corollary admit extensions to finding
grids in products $\R^{d_1}\times \dotsb\times \R^{d_k}$ where $k$ is greater than $2$
and $d_1,\dotsc,d_k$ are not necessarily all equal. The reader
is referred to Section~\ref{sec_grids} for the statement and the proof.

It is important to note that the corollary does not apply to a construction such
as \eqref{constr_kollar} in which the defining equation varies with the characteristic.
Our second result is a new construction of $K_{s,t}$-free graphs by means of hypersurfaces
with integer coefficients.
\begin{theorem}\label{thm_constr}
For every $s\geq 2$ there is an integer polynomial $f$ in $2s$ variables such that
the hypersurface $V=\{f=0\}\subset \C^s\times \C^s$ contains no $s$-by-$t$ grid and
no $t$-by-$s$-grid where $t=s^{4s}$.
\end{theorem}

Since the polynomial $f$ in the theorem has integer coefficients, it also
defines a hypersurface $V(\Fp)=\{f=0\}\subset \Fp^s\times\Fp^s$. Since $V$ contains
no $s$-by-$t$ grid, neither does $V(p)$ for all sufficiently large $p$.
Indeed, $s$-by-$t$ exists in $V$ only if a certain variety $U$ is non-empty, and
$s$-by-$t$ grid exists in $V(\Fp)$ only if $U(\Fp)$ is non-empty.

The proof of Theorem~\ref{thm_constr} uses two auxiliary constructions:
nondegenerate embeddings, and regular embeddings.
An embedding $f\colon \C^s\to \C^n$ is \emph{nondegenerate of order $t$} if
for any linear subspace $L\subseteq\R^n$ of codimension $s$
the intersection $f(\C^s)\cap L$ is finite and consists of at
most $t$ points.
\begin{theorem}\label{thm_nondeg}
For every $s$ and $n$ there exists a nondegenerate embedding $f\colon \C^s\to\C^n$ of order $(sn)^s$
where $f$ is a polynomial with integer coefficients of degree at most $sn$. Furthermore, $f$ can be 
taken to be a generic polynomial with integer coefficients of degree $sn$.
\end{theorem}
\textit{Remark.} After this paper was finished, nondegenerate embedding over finite fields 
were considered in \cite{evasive} under the name ``everywhere-finite varieties'', where
an algorithmic construction is given.

A map $f\colon \C^s\to \C^n$ is called \emph{$t$-regular} if $f(x_1),\dotsc,f(x_t)$ are linearly
independent for any distinct $x_1,\dotsc,x_t\in \C^s$. The definition of a regular embeddings
is not new, and it is known that $t$-regular embeddings exist for every $s$ and $t$ (see \cite{boltyanskii} and \cite{handel}).
The following lemma constructs regular embeddings that are polynomial of small degree.
\begin{lemma}\label{lem_reg}
For every $d\geq t-1$, and for any positive integers $s$ and $t$ there exists a $t$-regular 
embedding $f\colon \C^s\to \C^{(s+1)t}$ which is a polynomial of degree at most $d$ with integer coefficients. 
Furthermore, $f$ can be taken to be a generic polynomial of degree $d$.
\end{lemma}

Theorem \ref{thm_nondeg} and Lemma \ref{lem_reg} imply Theorem \ref{thm_constr}.
\begin{proof}[Proof of Theorem \ref{thm_constr}]
Let $f_1,f_2\colon \C^s\to \C^{s(s+1)}$ be generic polynomials of degree $d=s^2(s+1)$. Let $f\colon \C^s\times \C^s$ be
given by $f(x,y)=\langle f_1(x),f_2(y)\rangle$, where $\langle\, \cdot\, , \, \cdot\,\rangle$
denotes the inner product on $\C^{s(s+1)}$. Let $t=\bigl(s^2(s+1)\bigr)^s$. We show that $\{f=0\}$ 
contains no $s$-by-$t$ grid. It will follow by symmetry that it contains no $t$-by-$s$ grid either.

Note that $f_2$ is an $s$-regular embedding, and $f_1$ is a non-degenerate embedding of order $t$ 
by Lemma \ref{lem_reg} and Theorem \ref{thm_nondeg}, respectively. Let $S\subset \C^s$ be any
set of size $s$. Since $f$ is $s$-regular, the set
\[
  H=\{y : \langle x,y\rangle=0 \text{ for all }x\in S\}
\]
is a codimension $s$ subspace. Because $f_2$ is nondegenerate, it follows that
there are at most $t$ values or $y\in C^s$ such that $f_2(y)\in H$.
Since $\bigl(s^2(s+1)\bigr)^s\leq s^{4s}$ for $s\geq 2$, the theorem follows.
\end{proof}

\section{Grids on hypersurfaces}\label{sec_grids}
Let us consider a direct product of Euclidean spaces $P:=\R^{d_1}\times \R^{d_2}\times \cdots \times \R^{d_k}$ and a $k$-tuple $(a_1,\ldots,a_k)$ of positive integers.  
Any set of the form $A_1\times\cdots\times A_k\subset P$ such that 
\[
A_1\subset\R^{d_1},\ldots, A_k\subset\R^{d_k}\text{~~~and~~~}|A_1|=a_1,\ldots,|A_k|=a_k
\]
will be called a \textit{grid of size $(a_1,\ldots,a_k)$}. 
In this section we consider the following problem:

\begin{problem}
Find all pairs $(d_1,\ldots,d_k)$ and $(a_1,\ldots,a_k)$ of $k$-tuples of positive integers such that for every continuous function $\phi\colon P\to \R$ there exists a grid $A_1\times\cdots\times A_k$ of size $(a_1,\ldots,a_k)$ such that $\phi|_{A_1\times\cdots\times A_k}$ is a constant function; or equivalently there exist a fiber of $\phi$ that contains a grid of size $(a_1,\ldots,a_k)$.
\end{problem}

The space of all grids in $P$ can be modelled as the following product of configuration spaces
\[
\Gamma := F(\R^{d_1},a_1)\times\cdots\times F(\R^{d_k},a_k)
\]
where $F(\R^{d},a)$ denotes the space of all $a$-tuples of pairwise distinct points in $\R^{d}$, i.e.\
\[
F(\R^{d},a):=\{ (x_1,\ldots,x_a)\in (\R^d)^a ~:~ x_i\neq x_j\text{~for all }i<j \}.
\]
Any continuous map $\phi\colon P\to \R$ induces a continuous map
\[
\Phi : (\Gamma = F(\R^{d_1},a_1)\times\cdots\times F(\R^{d_k},a_k))\longrightarrow(R^{a_1\cdot \ldots \cdot a_k} \cong \R^{a_1}\otimes\cdots\otimes\R^{a_k})
\]
defined by
\[
\left( 
\begin{array}{ll}
(x_{1,1},\ldots,x_{1,a_1})     & \in F(\R^{d_1},a_1)  \\
(x_{2,1},\ldots,x_{2,a_2})     & \in F(\R^{d_2},a_2)  \\
\ldots                         & \ldots               \\
(x_{k,1},\ldots,x_{k,a_k})     & \in F(\R^{d_k},a_k)  
\end{array}
\right)
\longmapsto
\left(
\phi(x_{1,i_1},\ldots,x_{k,i_k})
\right)
_{(i_1,\ldots,i_k)\in [a_1]\times\cdots\times [a_k]}
\]
where $[\ell]:=\{1,\ldots,\ell\}$.

The symmetric group $\Sigma_{\ell}$ acts on the configuration space $F(\R^d,\ell)\subset (\R^d)^{\ell}$ and Euclidean space $\R^{\ell}$ by permuting coordinates. 
Having these actions in mind we can see that the map $\Phi:\Gamma\to\R^{a_1}\otimes\cdots\otimes\R^{a_k}$ is a $\Sigma_{a_1}\times\cdots\times\Sigma_{a_k}$-equivariant map if the action on 
\begin{itemize}
\item $\Gamma$ is assumed to be the componentwise, and on
\item $\R^{a_1}\otimes\cdots\otimes\R^{a_k}$ given by $(\pi_1,\ldots,\pi_k)\cdot (\mathbf{y}_1\otimes\cdots\otimes \mathbf{y}_k) := (\pi_1\cdot \mathbf{y}_1)\otimes\cdots\otimes (\pi_k\cdot \mathbf{y}_k)$, where $\pi_{\ell}\in\Sigma_{a_{\ell}}$ and $\mathbf{y}_{\ell}\in\R^{a_{\ell}}$.
\end{itemize}
In particular case when $d_1=\ldots = d_k$ and $a:=a_1=\ldots = a_k$ the function $\Phi$ can be considered as a $\Sigma_a\wr\Sigma_k$-equivariant map, where the actions on $\Gamma$ and $\R^{a_1}\otimes\cdots\otimes\R^{a_k}$ are appropriately extended.

Let $\Delta_{\ell}:=\{\mathbf{y}_{\ell}=(y_{\ell,1},\ldots,y_{\ell,a_{\ell}})\in\R^{a_{\ell}}~:~y_{\ell,1}=\cdots=y_{\ell,a_{\ell}}\}$.
The construction of the map $\Phi$ implies the following proposition:

\begin{proposition}
Let  $(d_1,\ldots,d_k)$ and $(a_1,\ldots,a_k)$ be $k$-tuples of positive integers.
\begin{enumerate}
\item If there is a continuous function $\phi\colon P\to \R$ such that no grid of size $(a_1,\ldots,a_k)$ is contained in any of its fibres, then there exists a $\Sigma_{a_1}\times\cdots\times\Sigma_{a_k}$-equivariant map
\[
\Gamma\longrightarrow\R^{a_1}\otimes\cdots\otimes\R^{a_k}{\setminus}\Delta_{a_1}\otimes\cdots\otimes\Delta{a_k}.
\]
\item If there is no $\Sigma_{a_1}\times\cdots\times\Sigma_{a_k}$-equivariant map
\[
\Gamma\longrightarrow\R^{a_1}\otimes\cdots\otimes\R^{a_k}{\setminus}\Delta_{a_1}\otimes\cdots\otimes\Delta{a_k}
\]
then every continuous function $\phi\colon P\to \R$ contains a grid of size $(a_1,\ldots,a_k)$ in at least one of its fibres.
\end{enumerate}
\end{proposition}

In the reset of this section we assume that $a_1=\ldots = a_k=p$ is an odd prime.
Let $\Fp[w_1,\ldots,w_k]$ be a polynomial ring with coefficients in the field $\Fp$ and with variables of degree $2$, i.e.\ $\deg w_1 =\ldots = \deg w_k=2$.
Consider the following polynomial 
\[
\theta := \Big(\prod_{(c_1,\dotsc, c_k)\in \Fp^k\setminus \{(0,\ldots,0)\}} (c_1w_1+c_2w_2+\dots +c_kw_k)\Big)^{\tfrac{1}{2}} \in \Fp[w_1,\ldots,w_k].
\]
We state the main result of this section that is a generalization of Theorem~\ref{thm_maintwo}:
\begin{theorem}
\label{prod-const}
Let $p$ be an odd prime, $a_1=\ldots = a_k=p$ and $(d_1, \dotsc, d_k)$ any $k$-tuple of positive integers. 
If the polynomial $\theta$ has a nonzero monomial $\alpha\cdot w_1^{i_1}\dots w_k^{i_k}$, $\alpha\in\Fp{\setminus}\{0\}$, with the property that
\[
2i_1 \leq (p-1)(d_1-1),\ \dotsc\ , 2i_k \leq (p-1)(d_k-1),
\]
then there is no $\Sigma_{p}\times\cdots\times\Sigma_{p}$-equivariant map
\begin{equation}
\label{eq_eq_map}
\Gamma\longrightarrow\R^{p}\otimes\cdots\otimes\R^{p}{\setminus}\Delta_{p}\otimes\cdots\otimes\Delta_{p},
\end{equation}
and consequently every continuous function $\phi\colon P\to \R$ contains a grid of size $(p,\ldots,p)$ in at least one of its fibres.
\end{theorem}

\begin{proof}
Let  $G:=\Z/p\times\cdots\times\Z/p=(\Z/p)^k$ be a subgroup the group $\Sigma_{p}\times\cdots\times\Sigma_{p}$ in the natural way.
Using the Fadell--Huseini index theory introduced in \cite{fadell_husseini} for the group $G$ with coefficients in the field $\Fp$ we prove that there is no $G$-equivariant map
\[
\Gamma\longrightarrow\R^{p}\otimes\cdots\otimes\R^{p}{\setminus}\Delta_{p}\otimes\cdots\otimes\Delta_{p}.
\]
This implies the non-existence of a $\Sigma_{p}\times\cdots\times\Sigma_{p}$-equivariant map \eqref{eq_eq_map} and therefore concludes the proof of the theorem.

\noindent Using the monotonicity property of the index \cite[page 74]{fadell_husseini}, under assumption on the polynomial $\theta$, we have to show that
\[
\mathrm{Index}_{G,\Fp}\Gamma \nsupseteqq \mathrm{Index}_{G,\Fp}\Big( \R^p\otimes\cdots\otimes\R^p{\setminus}\Delta_{p}\otimes\cdots\otimes\Delta_{p}\Big),
\]
where 
\[
\mathrm{Index}_{G,\Fp}X\subset H^*(G;\Fp):=\ker\Big(H^*(G;\Fp)\to H^*(\mathrm{E}G\times_{G}X;\Fp)\Big)
\]
denotes the Fadell--Husseini index of the $G$-space $X$.
The cohomology ring $H^*(G;\Fp)$ of the group $G$ is described by
\[
H^*(G;\Fp) = \Fp[w_1,\ldots,w_k]\otimes \Lambda [e_1,\ldots,e_k]
\]
where $\deg w_1 =\ldots = \deg w_k=2$, $\deg e_1 =\ldots = \deg e_k=1$ and $\Lambda[.]$ denotes the exterior algebra.
The coincidence of the notation with the definition of the polynomial $\theta$ is deliberate.

\noindent \textbf{1.} The tensor products $\R^p\otimes\cdots\otimes\R^p$ and $\Delta_{p}\otimes\cdots\otimes\Delta_{p}$, as real $G$-representations, are isomorphic as real $G$-representations with the group ring $\R [G]$ and its trivial sub-representation $\Delta$.
Therefore, there are $G$-equivariant deformation retractions
\[
\R^p\otimes\cdots\otimes\R^p{\setminus}\Delta_{p}\otimes\cdots\otimes\Delta_{p}\cong 
\R [G]{\setminus}\Delta\longrightarrow_{G}
I[G]{\setminus}\{0\}\longrightarrow_{G}
S(I[G]),
\]  
where $I[G]$ denotes the sub-representatiom $\{\sum_{g\in G}\alpha_g\cdot g~:~\sum_{g\in G}\alpha_g=0\}$. 
Consequently,
\[
\mathrm{Index}_{G,\Fp}\Big( \R^p\otimes\cdots\otimes\R^p{\setminus}\Delta_{p}\otimes\cdots\otimes\Delta_{p}\Big)=\mathrm{Index}_{G,\Fp}S(I[G]).
\]
Using the fact that the index of an $\Fp$-orientable sphere $S(V)$ of a $G$-representation $V$ is a principal ideal in the group cohomology ring generated by the Euler class of the vector bundle $V\to \mathrm{E}G\times_{G}V\to\mathrm{B}G$ along with the explicit computations in \cite{mann_milgram} and \cite[Lemma 2, page 364, english version]{volovikov_by_znp} we conclude that 
\begin{equation}
\label{eq_Index_1}
\mathrm{Index}_{G,\Fp}\Big( \R^p\otimes\cdots\otimes\R^p{\setminus}\Delta_{p}\otimes\cdots\otimes\Delta_{p}\Big)=\langle\theta\rangle.
\end{equation}
What we actually used is that the Euler class of the representation $I[G]$ is $\theta$.

\noindent \textbf{2.} Let us determine the index of the configuration space $F(\R^d,p)$ with respect to the action by cyclic shifts of the cyclic group $\Z/p$.
The Serre spectral sequence of the fibration $F(\R^d,p)\to \mathrm{E}\Z/p\times_{\Z/p}F(\R^d,p)\to \mathrm{B}\Z/p$ has the $E_2$-term given by
\[
E_2^{i,j}=H^i(\Z/p;H^j(F(\R^d,p);\Fp))
\] 
and converges to $H^*(\mathrm{E}\Z/p\times_{\Z/p}F(\R^d,p);\Fp)$.
Therefore, $\mathrm{Index}_{\Z/p,\Fp}F(\R^d,p)$ can be obtained from the spectral sequence in the following way:
\[
\mathrm{Index}_{\Z/p,\Fp}F(\R^d,p)=\bigcup_{r\geq 2}\ker\Big( E_r^{\ell,0}\to E_{r+1}^{\ell,0}\Big).
\]
In \cite[Theorem 8.2, page 268]{cohen} Cohen describes the $E_2$-term of this spectral sequence showing that
\[
E_2^{i,j}=0\text{   for all   }1\leq j\leq (d-1)(p-1)-1 \text{  and  }i\geq 1 .
\]
Since the differentials of this Serre spectral sequence are $H^*(\Z/p;\Fp)$-module maps and cohomology of the cyclic group $H^*(\Z/p;\Fp)=\Fp[w]\otimes \Lambda[e]$ has an element $w$ that is not a divisor of zero we conclude that all the differentials $\partial_r$ for $2\leq r\leq (d-1)(p-1)$ that land in the $0$-row are trivial.
In other words, the maps $E_r^{\ell,0}\to E_{r+1}^{\ell,0}$ are injective for $0\leq\ell\leq (d-1)(p-1)$ and all $r\geq 2$.
Therefore
\begin{equation}
\label{eq_Index_2}
\mathrm{Index}_{\Z/p,\Fp}F(\R^d,p)\subseteq H^{\geq (d-1)(p-1)+1}(\Z/p;\Fp).
\end{equation}

\noindent \textbf{3.} Let us denote by $H_r:=\Z/p\times\cdots\times \{0\}\times\cdots\times\Z/p$, for $1\leq r\leq k$ subgroups of the group $G$ that are obtained by substituting the $r$-th summand in the product with the trivial group.
Moreover, let $K_r:=\{0\}\times\cdots\times \Z/p\times\cdots\times\{0\}$ be a subgroup of $G$ such that $G=K_r\oplus H_r$.
Without losing generality we can assume that 
\[
H^*(K_r;\Fp)=\Fp[w_r]\otimes \Lambda[e_r]\text{~and~}H^*(H_r;\Fp)=\Fp[w_1,\ldots,\widehat{w}_r,\ldots,w_k]\otimes \Lambda[e_1,\ldots,\widehat{e}_r,\ldots,e_k]
\]  
where $~\widehat{ }~$ means deleting this term from the list.
Then by the previously established inclusion \eqref{eq_Index_2} and \cite[Proposition 3.1, page 75]{fadell_husseini} we have that
\begin{equation}
\label{eq_Index_3}
\mathrm{Index}_{G,\Fp}\Gamma \subseteq H^{\geq N}(K_1;\Fp)\otimes H^*(H_1;\Fp)+\cdots + H^{\geq N}(K_k;\Fp)\otimes H^*(H_k;\Fp)
\end{equation}
where $N=(d-1)(p-1)+1$.

\noindent \textbf{4.} Under the assumption of the theorem on the polynomial $\theta$ the relation \eqref{eq_Index_3} implies that
\[
\theta\notin \left\langle w_1^{\frac{(d-1)(p-1)}{2}+1},\ldots ,w_k^{\frac{(d-1)(p-1)}{2}+1}\right\rangle \supseteq\mathrm{Index}_{G,\Fp}\Gamma \cap \Fp[w_1,\ldots ,w_k]
\]
and consequently
\[
\mathrm{Index}_{G,\Fp}\Gamma \nsupseteqq \mathrm{Index}_{G,\Fp}\Big( \R^p\otimes\cdots\otimes\R^p{\setminus}\Delta_{p}\otimes\cdots\otimes\Delta_{p}\Big).
\]
Therefore, there can not exist a $G$-equivariant map
\[
\Gamma\longrightarrow\R^{p}\otimes\cdots\otimes\R^{p}{\setminus}\Delta_{p}\otimes\cdots\otimes\Delta_{p},
\]
and the theorem is proved.
\end{proof}

\medskip

\begin{proof}[Proof of Theorem \ref{thm_maintwo}]
In the case when $k=2$ the polynomial $\theta$ can be presented in the following way
\begin{align*}
\theta^2 &= (p-1)! w_2^{p-1} \prod_{i\in \Fp{\setminus\{0\}}, j\in \Fp} (iw_1 + jw_2)
= w_2^{p-1} \prod_{k\in \Fp} (w_1 + k w_2)^{p-1}
\\&= w_2^{p-1} (w_1^p - w_1w_2^{p-1})^{p-1} =w_1^{p-1}w_2^{p-1}(w_1^{p-1} - w_2^{p-1})^{p-1},
\end{align*}
where we used the equality $\prod_{k\in \Fp} (t + k) = t^p - t$ modulo $p$. 
Therefore,
\[
\theta = w_1^{\frac{p-1}{2}} w_2^{\frac{p-1}{2}} (w_1^{p-1} - w_2^{p-1})^{\frac{p-1}{2}}.
\]
For brevity put $m = \frac{p-1}{2}$.
The monomials of the polynomial $\theta$ that have nonzero coefficients are of the form
\[
w_1^{m + 2m \ell}w_2^{m + 2m (m-\ell)},\quad\text{ where~~} 0\leq \ell\leq m.
\]
The sufficient conditions of Theorem~\ref{prod-const} for the existence of the grid applied in this $k=2$ case imply that 
\[
d_1 \geq 2\ell + 2,~~~ d_2 \geq 2 (m-\ell) + 2\quad\text{ for some~~}\ 0\leq \ell\leq m.
\]
In particular, if $m$ is divisible by $2$ (i.e., $p-1$ is divisible by $4$) we may take $d_1 = d_2 = m+2 = \frac{p+3}{2}$.
The case $p=5$, $d=4$ is an example of this type.
On the other hand, for $m$ odd we can take $d_1 = d_2 = m+3 = \frac{p+5}{2}$.
The formula that unifies both cases is given in Theorem~\ref{thm_maintwo}.
\end{proof}

\section{Nondegenerate embeddings}
In this section we give two constructions of nondegenerate embeddings. First,
we give nonconstructive proof for Theorem~\ref{thm_nondeg}. Then, we shall
present slightly inferior, but explicit construction.

\begin{proof}[Proof of Theorem~\ref{thm_nondeg}]
We shall construct a sequence of polynomials $f_1,f_2,\dotsc \colon \C^s\to \C$ such that
$\deg f_i\leq si$ and for every $n$ the map $f=(f_1,f_2,\dotsc,f_n)$ from $\C^s$ into $\C^n$
is nondegenerate of order $s$. Each of the polynomials $f_1,f_2,\dotsc$ will be of degree
at most $sn$.

The first $s$ functions $f_1,\dotsc,f_s$ are any $s$ algebraically independent polynomials. We could take
them to be projections on the respective coordinates, or generic polynomials.

Suppose we have chosen $f_1,\dotsc,f_{n-1}$.
For $f_n$ we shall choose a generic polynomial of degree $sn$. The preimage of a
codimension $s$ subspace under $f$ is the variety $V$ given by the system of $s$ equations
\begin{equation}\label{linsys}
V=\left\{\sum_{j=1}^n a_{i,j} f_j=0,\qquad\text{for }i=1,\dotsc,s\right\},
\end{equation}
where $A=\{a_{i,j}\}_{i,j}$ is an $s$-by-$n$ matrix of rank~$s$. We can assume that the system is in echelon
form, i.e., there are numbers $c_1<\dotsb<c_s$ such that $a_{i,j}=0$ if $j>c_i$ and $a_{i,j}=1$ if $j=c_i$.
We shall call the set $\sh(A)\eqdef\{c_1,\dotsc,c_s\}$ the shape of $A$. Alternatively, the shape
of $A$ is the Schubert cell of $A$.

We shall show that the variety $V$ is zero-dimensional. It will then follow by
Bezout's theorem that the number of the solutions is at most
$\deg V\leq (\max \deg f_i)^s\leq (sn)^s$. Let $g_i=\sum_{j=1}^{n-1} a_{i,j}f_j$
(note that the summation ranges to $n-1$, not $n$). Put $U=\{g_1=\dotsb=g_{s-1}=0\}$, so that
$V=U\cap \{g_s+f_n=0\}$. Since $(f_1,\dotsc,f_{n-1})\colon \C^s\to\C^{n-1}$ is a nondegenerate
embedding, it follows that $\dim U=1$.

Let $\lambda$ be a possible shape for a matrix $A$. Now fix $\lambda$.
Denote by $P_{d,s}$ the space of all degree $d$ polynomials on $\C^s$.
Let
\[B_\lambda=\{ \bar{f}\times A \in P_{d,s} \times M_{s\times n}(\C) :  \sh(A)=\lambda \text{ and } \dim (V \cap \{g_s+\bar{f}=0\})=1 \} \]
be the set of ``bad'' polynomials $f_n$. The set $B$ is a variety, and we shall estimate its dimension
by bounding dimension of the fibers $B_\lambda\cap \{A=A_0\}$. Fix a value of $A$, and hence the variety $U$.
Let $U_1,\dotsc,U_m$ be the irreducible components of $U$. Let
\[B_{\lambda,A}(i)=\{\bar{f} \times P_{d,s} : \{g_s+\bar{f}=0\}\supset U_i\}.\]
Since $\dim M_{s\times n}(\C)=sn$, the fibers of $B_\lambda$ are unions of $B_{\lambda,A}$'s,
it suffices to show that $\dim B_{\lambda,A}(i)<\dim P_{d,s}-sn$. Since $U_i$ is irreducible, it follows that
$\{g_s+\bar{f}=0\}\supset U_i$ is equivalent to $g_s+\bar{f}\in I(U_i)$, where $I(U_i)$
is the ideal of $U_i$. Since $\dim I(U_i)=\dim U_i=1$, the Hilbert function of $I(U_i)$ satisfies
$H(I(U_i); d)\geq d$, which means that $\dim B_{\lambda,A}(i)<\dim P_{d,s}-sn$ if $d\geq sn$.
\end{proof}

Let $\C^*=\C\setminus \{0\}$. We now give an explicit nondegenerate polynomial
embeddings of $(\C^*)^s$ into $\C^n$. For that we first state a theorem by Bernstein.
A \emph{Laurent polynomial} $f\in \C[x_1,x_1^{-1},\dotsc,x_d,x_d^{-1}]$ is a linear
combination of monomials $x^s=x_1^{s_1}\dotsb x_n^{s_d}$. The \emph{support} $\supp f$ of a polynomial
$f=\sum c_s x^s$ is the set of all $s$ such that $c_s\neq 0$.
For a non-zero rational vector $\alpha=(\alpha_1,\dotsc,\alpha_d)$
let $m(\alpha,S)=\min \{\langle \alpha, s \rangle : s\in S\}$,
and $S_\alpha=\{s\in S : \langle \alpha, s\rangle = m(\alpha,S)\}$.
Then for a Laurent polynomial $f=\sum c_s x^s$ with $\supp f\subset S$ let
$f_\alpha=\sum_{s\in S_\alpha} c_s x^s$. For sets $S,T\subset \Z^n$,
the notation $S+T\eqdef \{s+t : s\in S, t\in T\}$ denotes the Minkowski
sum of $S$ and $T$. Finally, for sets $S_i \subset \Z^n$ let
\begin{equation}\label{eq_defmixvol}
V(S_1,\dotsc,S_d)=\sum_{\substack{I\subset [d]\\I\neq \emptyset}} (-1)^{d-\abs{I}} \vol\left(\conv\Bigl(\sum_{i\in I} S_i\Bigr)\right)
\end{equation}
denote the mixed volume of the Newton polytopes of $S_i$'s. We remark that
the underlying meaning of the mixed volume is that $\vol(\lambda_1\cdot C_1+\dotsb+
\lambda_d\cdot C_d)$ is a polynomial in $\lambda_1,\dotsc,\lambda_d\in \R^+$ whenever $C_1,\dotsc,C_d$
are compact convex sets.

\begin{lemma}[Theorem B in \cite{bernstein}]
Suppose $f_1,\dotsc,f_d$ are Laurent polynomials in $d$ variables, and $S_i=\supp f_i$.
If for every non-zero rational vector $\alpha$ the system $f_{1\alpha}=\dotsb=f_{d\alpha}=0$
has no zeros on $(\C^*)^d$, then the system $f_1=\dotsb=f_d=0$ has exactly $V(S_1,\dotsc,S_d)$ solutions
on $(\C^*)^d$.
\end{lemma}
\begin{theorem}
Let $f_i=\sum_{j=1}^s x_j^{p_{ij}}+x_j^{-p_{ij}}$ be a polynomial in variables $x_1,\dotsc,x_s$.
Suppose the exponents $p_{ij}$ are distinct prime numbers satisfying $p_{ij}<p_{i'j}$ for every $i<i'$.
Then $(f_1,\dotsc,f_n)\colon (\C^*)^s\to \C^n$ is a nondegenerate embedding. Furthermore,
its order is at most $(4^s/s!) (s\max p_{ij})^s$.
\end{theorem}
\begin{proof}
As above it suffices to show that the system of equations \eqref{linsys} has only finitely many solutions.
As before we assume that the system is in the echelon form, and of shape $\{c_1,\dotsc,c_s\}$, and
define $g_i=\sum_{j=1}^{n-1} a_{i,j}f_j$. Let $\alpha$ be a non-zero rational vector.
Since the Newton polytopes of $f_1,\dotsc,f_n$ are nested crosspolytopes, $\supp g_{i\alpha}\subset \supp f_{c_i}$.
If $g_{i\alpha}$ is a sum of powers $x_{j_1},\dotsc,x_{j_k}$, then $\alpha_{j_l}/\alpha_{j_m}=p_{c_i,j_m}/p_{c_i,j_l}$.
To each $i$ associate an arbitrary spanning tree $T_i$ on the ground set $\{j_1,\dotsc,j_k\}$. The union
$\bigcup_{i=1}^s T_i$ cannot have a cycle because it would correspond to a non-trivial product
of primes and their inverses being equal to $1$. Thus $\bigcup_{i=1}^s T_i$ is a tree, and so has at most $s-1$ edges.
By the pigeonhole principle one of $T_i$'s has no edges, and corresponding equation $g_{i\alpha}=0$ has only one term.
Since that equation has no solutions in $(\C^*)^s$, the conditions of Bernstein's theorem hold for \eqref{linsys}.

Since the unit crosspolytope has volume $2^s/s!$, each term in the alternative sum \eqref{eq_defmixvol}
is at most $2^s/s! (s\max p_{ij})^s$. Since there $2^s-1$ terms, by Bernstein's theorem the number
of solutions is at most
\[
  (4^s/s!)(s\max p_{ij})^s\leq (4e \max p_{ij})^s.\qedhere
\]
\end{proof}

If we choose $p_{ij}$'s from the first $sn$ prime numbers, and use that $m$'th prime
number is less than $2m\log m$ for $m\geq 3$, from the theorem above we obtain an
explicit nondegenerate embedding of order less than
$\bigl(8e sn\log(sn)\bigr)^s$.

\section{Polynomial regular embeddings}
In this section we construct polynomial regular embeddings.
\begin{proof}[Proof of Lemma~\ref{lem_reg}]
We shall construct $f$ in two steps. First, we construct a $t$-regular polynomial $f$
which maps to $\C^{\binom{s+d}{d}}$. Then we shall compose it with a suitable projection $\C^{\binom{s+d}{d}}\to \C^{(s+1)t}$
to obtain the required map.

Let $B_{d,s}$ be the set of all monomials of degree at most $d$ on $\C^s$.
Define $f\colon \C^s\to \C^{B_{s,d}}$ via $f_b(x)=b(x)$ for every element $b\in B_{d,s}$. The
function $f$ is the so-called Veronese embedding.
It is $t$-regular. Indeed, suppose $x_1,\dotsc,x_t$ are any $t$ distinct elements
$x_1,\dotsc,x_t\in \C^s$ and that $f(x_1),\dotsc,f(x_t)$ are linearly dependent, i.e., there
are scalars $\alpha_1,\dotsc,\alpha_t$ such that $\sum \alpha_i f(x_i)=0$. The latter
condition is equivalent to
\begin{equation}\label{eq_vanish}
\sum \alpha_i g(x_i)=0 \text{ for every polynomial $g$ of degree at most $d$}.
\end{equation}
Let $p\colon \C^s\to \C$ be projection such that $p(x_1),\dotsc,p(x_t)$ are distinct.
Then one of the symmetric forms $\sum \alpha_i p(x_i)^j$ must be non-zero,
contradicting \eqref{eq_vanish}.

Let $n=(s+1)t$. Let $\Prj$ be the space of all linear maps from $\C^{B_{s,d}}$ to
$\C^n$. For $\bx=(x_1,\dotsc,x_t)\in (\C^s)^t$ set
\[B(\bx)=\bigl\{ p\in \Prj : p(f(x_1)),\dotsc,p(f(x_t)) \text{ are linearly dependent}\bigr\},\]
and
\[B=\bigcup_{\bx} B(\bx).\]
Since $f(x_1),\dotsc,f(x_t)$ are linearly independent, the codimension of $B(\bx)$ is
\[
\dim \Prj-\dim B(\bx)=n-(t-1).
\]
As $\bx$ ranges over $st$-dimensional space and $st<n-(t-1)$,
it follows that $\dim B<\dim \Prj$, i.e., for a generic projection $p$ the map $p\circ f$ is $t$-regular.
The composition $p\circ f$ is a generic polynomial of degree $d$.
\end{proof}\vspace{2ex}

\textbf{Acknowledgement.} We wish to thank Oberwolfach Insitute, where this work began.
We also thank Zolt{\'a}n F{\"u}redi and Simeon Ball for helpful comments.

\bibliographystyle{alpha}
\bibliography{Submission}

\end{document}